\documentclass{amsart} \usepackage{amssymb,amsfonts,amscd}
\newtheorem{theorem}{Theorem}[section]
\newtheorem{lemma}[theorem]{Lemma}
\newtheorem{corollary}[theorem]{Corollary}
\newtheorem{proposition}[theorem]{Proposition}
\theoremstyle{remark}

\theoremstyle{definition}

\numberwithin{equation}{section} \makeatother

\DeclareMathOperator{\spn}{Span} 
 
\DeclareMathOperator{\Kdb}{{\mathbb K}}

\DeclareMathOperator{\Cdb}{{\mathbb C}}
\DeclareMathOperator{\Rdb}{{\mathbb R}}

\DeclareMathOperator{\Tdb}{{\mathbb T}}
\DeclareMathOperator{\Ndb}{{\mathbb N}}

\begin{document}

\title[Operator algebras with cai II]{Operator algebras with contractive approximate identities II}
\author{David P. Blecher}
\author{Charles John Read}

\address{Department of Mathematics, University of Houston, Houston, TX
77204-3008}
 \email[David P.
Blecher]{dblecher@math.uh.edu}
\date{Revision of November 15, 2012}
\address{Department of Pure Mathematics,
University of Leeds,
Leeds LS2 9JT,
England}
 \email[Charles John Read]{read@maths.leeds.ac.uk}
\thanks{*Blecher was partially supported by a grant  from
the National Science Foundation.}

\begin{abstract}
We make several contributions to our recent program investigating structural properties of algebras of operators on a Hilbert space.  For example, we make substantial contributions to the  noncommutative peak interpolation program begun by Hay and the first author, Hay and Neal.  Another
sample result: an operator algebra has a  contractive approximate identity iff the linear span of the elements with positive real part is dense.   We also extend the theory of compact projections to the most general case.  Despite the title, our algebras are often allowed to have 
no approximate identity.
\end{abstract}


\maketitle

\section{Introduction}  For us, an operator algebra is a norm closed algebra of operators on a Hilbert space.
In a long series of papers (e.g.\ \cite{BHN,ABS,BRead,BNII,ABR,Read}), we and coauthors have investigated 
structural properties of approximately unital 
operator algebras  in terms of certain one-sided ideals, hereditary subalgebras
(HSA's), noncommutative topology (open and closed projections), etc.  Most recently, we have proposed the set 
$\frac{1}{2} {\mathfrak F}_A = \{ x \in A : \Vert 1 - 2 x \Vert_{A^1} \leq 1 \}$,
where $A^1$ is the unitization of an operator algebra $A$,
 as the analogue of the positive part of the unit ball of a 
$C^*$-algebra, and this approach is proving fruitful \cite{BRead, BNI, BNII}.
In the present paper we advance this program in 
  several ways.
For example, in Sections 2 and 5 we tackle two important and related
topics.  First, we show that not all right
ideals having a contractive approximate identity (cai) 
 in a unital operator algebra are proximinal (we recall that $J$ is 
proximinal in $A$ if
the distance $d(x,J)$ is achieved for all $x \in A$).  Our example is quite simple, and 
solves a  question that dates to the time of \cite{BEZ} (we recall the 
useful fact that two-sided ideals with cai, and more generally all $M$-ideals, 
are proximinal \cite{HWW}; as are closed one-sided ideals in $C^*$-algebras (see e.g.\ \cite{Brown,Ki}), and this
is extremely important, as may be seen for example in the use 
made of proximinality in the latter references).  
As a complement, we give a
very natural  sufficient condition that 
does ensure best approximation in such a right ideal.   These allow us to make some substantial advances, and also to  put some boundaries  in place, in the 
 noncommutative peak interpolation
 program begun by Hay in his Ph.D.\ thesis
\cite{Hayth}, and the first author, Hay and Neal \cite{BHN,Hay}. 
That is, we close in on the possible noncommutative generalizations of classical facts
such as: if $A$ is a 
function algebra on a compact space $K$ and if $E$ is {\em peak set} in $K$ (defined below), then the
functions  in $C(E)$ which are restrictions  of functions in $A$ to $E$, have norm preserving
extensions in $A$.  More generally, if $f$ is
a strictly positive function on $K$, then the continuous functions  on $E$  which are restrictions of
functions in $A$, and which are dominated
by $f$ on $E$, have extensions $h$ in $A$ satisfying $|h| \leq f$ on all of $K$ 
(see e.g.\ II.12.5 in \cite{Gam}).  In particular we give (in Sections 2, 5 and 6)
several new noncommutative peak interpolation results which
generalize such function algebra results.   We remark that after this paper
had been refereed we discovered one of our best
noncommutative peak interpolation result, a generalization of the just mentioned
classical result involving $h$ and $f$,  which we state
here as Theorem \ref{babint}.
  Moreover the new idea in the latter result yields  
short proofs \cite{Bnew} both of Read's theorem on approximate identities \cite{Read}, and
Hay's main theorem in \cite{Hay,Hayth}.   These are absolutely foundational results in
the subject, and the extreme depth of their proofs hindered their accessibility 
until now. 

In Section 3 we prove some results that should allow many facts from  
\cite{BRead, BNI, BNII}  to be generalized
by widening the set 
 $\frac{1}{2} {\mathfrak F}_A$  to ${\mathfrak r}_A = \{ a \in A :  {\rm Re}(a) =  a  +a^* \geq 0 \}$.
This will have implications to our study of operator algebras in terms of a new kind of positivity,
and we are investigating this in independent works.   This does not mean that the ${\mathfrak F}_A$
approach has been superseded;  however using   ${\mathfrak r}_A$ instead
will be adequate and simpler for certain applications.
In Section 4 we present a mechanism to 
extend some of our results on  approximately unital operator 
algebras to algebras with no kind
of approximate identity.   Indeed such an algebra has a HSA 
containing all other subalgebras having a contractive approximate identity
(cai), and this HSA is the closure of the {\em linear} span of $\frac{1}{2} {\mathfrak F}_A$.
 Thus an operator algebra $A$ has a cai  
 iff the linear span of ${\mathfrak F}_A$ or  ${\mathfrak r}_A$ is dense.  In Section 6 we turn
to noncommutative topology, and generalize the compact projections in the 
sense of the recent paper \cite{BNII}, and their theory, to algebras with no approximate identity.  We also identify the
right  ideals which correspond to
 compact projections.    

As with \cite{BRead}, many of our results apply immediately to 
give new results for function algebras, but we will not take the time
to point these out. 
 
We now turn to notation and more precise definitions.
The reader is referred for example to \cite{BLM,BHN,BRead} for more details on
some of the topics below if needed.   By an {\em ideal} of an operator algebra $A$ we shall
always mean a closed two-sided ideal in $A$.   For us a {\em projection}
is always an orthogonal projection.
 The letter $H$ is reserved for Hilbert spaces.
If $E, F$ are sets, then $EF$ denotes the
closure of the span of products of the form $xy$ for $x \in E, y \in F$.   
We often use silently the fact from basic analysis that
$X^{\perp \perp}$ is the weak* closure in $Y^{**}$ of
a subspace  $X \subset 
Y$.   We recall that by a theorem due to  Ralf Meyer,
every operator algebra $A$ has a unique unitization $A^1$ (see
 \cite[Section 2.1]{BLM}). Below $1$ always refers to
the identity of $A^1$ if $A$ has no identity (we set $A^1 = A$ if $A$ has an identity of norm $1$).  If $A$ is a nonunital
operator algebra represented (completely) isometrically on a Hilbert
space $H$ then one may identify $A^1$ with $A + \Cdb I_H$.   The
second dual $A^{**}$ is also an operator algebra with its (unique)
Arens product, this is also the product inherited from the von Neumann
algebra $B^{**}$ if
$A$ is a subalgebra of a $C^*$-algebra $B$.     The `meet' or `join' in
$B^{**}$ of projections in $A^{**}$ remains in $A^{**}$ 
(since for example if $p,q$ are projections in $A^{**}$
then it is well known that $p \wedge q$ (resp.\ $p \vee q$) is a weak* limit of $(pq)^n$
(resp.\ of polynomials in $p$ and $q$ with no constant term), which stays inside 
the von Neumann algebra
$\{ x \in A^{\perp \perp} : x^* \in A^{\perp \perp} \} \subset B^{**}$).   Note that
$A$ has a  contractive approximate identity (cai) iff $A^{**}$ has an identity $1_{A^{**}}$ of norm $1$,
and then $A^1$ is sometimes identified with $A + \Cdb 1_{A^{**}}$.
Often our algebras 
and ideals  are {\em approximately unital}, i.e.\ have a cai; when this
is the case we will say so.

We recall that an {\em r-ideal} is a right ideal with a left cai, and an {\em $\ell$-ideal} is a left ideal with a right cai.
We say that an operator algebra $D$ with cai, which is a subalgebra of
another operator algebra $A$, is a HSA (hereditary subalgebra)
in $A$, if $DAD \subset D$.    See
\cite{BHN} for the theory of HSA's (a few more results may be found in \cite{ABS,BRead}).
HSA's in $A$ are in an order preserving,
bijective correspondence with the r-ideals in $A$, and 
 with the $\ell$-ideals in $A$.   Because of this symmetry 
we will usually restrict our results to the r-ideal case; the $\ell$-ideal case will be analogous.
There is also a bijective correspondence with the  
{\em open projections} $p \in A^{**}$, by which we mean that there
is a net $x_t \in A$ with $x_t = p x_t  \to p$ weak*, 
or equivalently with $x_t = p x_t p \to p$ weak* (see  \cite[Theorem 2.4]{BHN}).  These are
also the open projections $p$ in the sense of Akemann \cite{Ake2} in $B^{**}$, where $B$ is a $C^*$-algebra containing $A$, such that
$p \in A^{\perp \perp}$.   If $A$ is approximately unital then the complement $p^\perp = 1_{A^{**}} - p$
 of an open projection for $A$
 is called a {\em closed projection} for $A$.  We spell out some of the correspondences above:
if $D$ is a HSA in $A$, then $DA$ (resp.\ $AD$) is the matching  r-ideal (resp.\ $\ell$-ideal),
and $D = (DA)(AD) = (DA) \cap (AD)$.
The weak* limit of a cai for $D$, or of a left cai for an r-ideal, is
an open projection, and is called the {\em support projection}.  
Conversely, if $p$ is an open projection in $A^{**}$, then
$pA^{**} \cap A$ and $pA^{**}p \cap A$ is 
the matching r-ideal and HSA pair
in $A$.    It is  well known that the closure $J$ of a sum of r-ideals $J_i$ is an r-ideal,
but for the readers convenience we include a proof (using basic facts from e.g.\ \cite[Section 2.5]{BLM}).
Obviously $J$ is  a right ideal, and $J^{\perp \perp}$ is the weak* closure $K$ of the span of $e_i A^{**}$, 
where $e_i$ is the  support projection of $J_i$.  If $p = \vee_i \, e_i$, then
 $K \subset p A^{**}$ since $e_i \leq p$.
Conversely, $p \in J^{\perp \perp}$ by a remark above about
meets and joins, and since $J^{\perp \perp}$ is a weak* closed right ideal
we have $p A^{**} \subset K$.  So $J^{\perp \perp} = p A^{**}$ has left identity $p$
and it follows from e.g.\ \cite[Proposition 2.5.8]{BLM} that $J$ is an r-ideal (with support projection
$p$).  
 
Let $A$ be an operator algebra.   If $x \in A$ then oa$(x)$ denotes the closed subalgebra 
of $A$ generated by $a$. The set
${\mathfrak F}_A = \{ x \in A : \Vert 1 - x \Vert \leq 1 \}$ equals
$\{ x \in A : \Vert 1 - x \Vert_{A^1} = 1 \}$ if $A$ is nonunital, whereas
if $A$ is unital then ${\mathfrak F}_A = 1 + {\rm Ball}(A)$.  Many properties of ${\mathfrak F}_A$ are 
developed in \cite{BRead}.  An $x \in A$ is in ${\mathfrak F}_A$ iff $x + x^* \geq x^* x$;
and $x \in \Rdb^+  {\mathfrak F}_A$ iff there is a constant $C > 0$ with  $x + x^* \geq C x^* x$.
 If $A$ is a closed subalgebra of an operator algebra $B$
then it is easy to see, using the
uniqueness of the unitization, that ${\mathfrak F}_A = A \cap {\mathfrak F}_B$.  
 If $x \in \frac{1}{2} {\mathfrak F}_A$
then $\Vert x \Vert \leq 1$ and $\Vert 1- x \Vert \leq 1$.   The following
is a slight rephrasing of \cite[Theorem 2.12]{BRead}, and it follows from its proof (we remark that there 
is a typo in the parallelogram law stated in that proof, the quantity after the `$=$' should be 
multiplied by $2$):

\begin{lemma}  \label{abr}  Let $T$ be an operator in
 $B(H)$ with $\Vert I - T \Vert \leq 1$.  Then  $T$ is not invertible if
and only if
$\Vert I - \frac{1}{2} T \Vert = 1$.
Also, $T$ is invertible iff $T$ is invertible in the closed
algebra generated by $I$ and $T$, and iff  {\rm oa}$(T)$ contains $I$.
Here $I = I_H$.
 \end{lemma}

Generalizing Akemann's $C^*$-algebraic notion (see e.g.\ \cite{Ake,AP}), a {\em compact} projection for an approximately 
unital operator algebra $A$  is a closed projection in $A^{**}$ with $q = qa$ for some $a \in {\rm Ball}(A)$
(this may then  be done with $a \in  \frac{1}{2} {\mathfrak F}_A$).  See \cite{BNII}.
We recall that if $A$ is a
space of continuous functions on a compact set $K$, then a closed
set $E \subset K$ is called a {\em peak set}
 if there exists $f \in A$ such that $f|_E = 1$ and
$|f(x)| < 1$ for all $x \notin E$.   These sets have been generalized to 
operator algebras in \cite{Hayth,Hay,BHN,BRead,BNII}.  There are
various equivalent definitions of peak projections in 
the latter papers (we warn the reader that if $A$ is not unital then the 
definition late in \cite[Section 2]{BRead} is not equivalent and will not be 
considered here).   If $a \in {\rm Ball}(A)$ define $u(a)$ to be the weak* limit
of $a (a^* a)^n$.  If this is a projection then this also equals
the weak* limit $\lim_n \, a^n$ (see \cite[Lemma 3.1]{BNII}), and in this case
we will call $u(a)$ a  {\em peak projection} for $A$
 and say that $a$ {\em peaks at} $u(a)$.

\begin{lemma} \label{haye}  Let $x \in {\rm Ball}(B)$ for a $C^*$-algebra $B$,
  and let $q$ be a closed projection in $B^{**}$ such that $xq=q$.  The following
   are equivalent:
  \begin{enumerate}
 \item $x$ peaks at $q$ in the sense above (that is, $q = u(x)$),
   \item $\| px \| < 1$ for every closed projection $p$ in $B^{**}$
with $p \le 1-q$,
   \item $\| xp \| < 1$ for every closed projection $p \le 1-q$.
 \end{enumerate}
\end{lemma}

\begin{proof}   Item (2) is the same as (5) in
\cite[Lemma 3.1]{BNII}, but with the words `closed' and `compact' interchanged.
However if one traces through the proof that (1) implies (5) there, one sees that one only
really used that $p$ is closed.  Thus (1) is equivalent to (2).  Since
(1) is a symmetric condition, and $xq=q$ iff $qx = q$ since $\Vert x \Vert \leq 1$,
  we must have (1)  equivalent to (3) too.
\end{proof}

\begin{lemma}  For any operator algebra $A$, the peak projections for $A$ 
are the weak* limits of $a^n$ for $a \in {\rm Ball}(A)$ in the cases that such 
limit exists.  \end{lemma}  

\begin{proof}  The one direction is proved in 
 \cite[Lemma 3.1]{BNII}, which is a slight generalization of 
Hay's unital variant of that result \cite{Hay}.
Conversely, if the weak* limit $\lim_n \, a^n$ exists, then it is
a closed projection $q$, indeed it is a peak projection,
 with respect to $A^1$ by \cite[Corollary 6.9]{BHN}.  By that result $q = u(\frac{1+a}{2})$,
and by the last lines of the proof of \cite[Theorem 3.4 (2)]{BNII}, $q = u(x)$ for $x = \frac{a+a^2}{2}
\in {\rm Ball}(A)$.
\end{proof} 
 
If $a \in \frac{1}{2} {\mathfrak F}_A$
then the weak* limit $\lim_n \, a^n$ exists, and equals $u(a)$ (see 
\cite[Corollary 3.3]{BNII}).
 If $A$ is unital then $u(a) = s(1-a)^\perp$ and
$s(a) = u(1-a)^\perp$
for $a \in \frac{1}{2} {\mathfrak F}_A$,
where $s(\cdot)$ is the {\em support projection}
from \cite[Section 2]{BRead} (see \cite[Proposition 2.22]{BRead}).
Compact projections for an approximately 
unital operator algebra are just the  decreasing limits of peak projections
\cite[Theorem 3.4]{BNII}.  They are also the projections in $A^{**}$ 
 compact  with respect to any $C^*$-algebra containing $A$, or 
 with respect to $A^1$,
 by e.g.\ \cite[Theorem 2.2]{BNII}. 
If $A$ is separable and 
unital (resp.\ approximately 
unital) then every closed (resp.\ compact)
projection in $A^{**}$ is a
 peak projection (see \cite[p.\ 200]{BRead} 
and \cite[Theorem 3.4]{BNII}).   In Section 6 we generalize these facts.

\section{Proximinality and noncommutative peak interpolation}  

Peak interpolation focuses on constructing functions in a given algebra
of functions on a topological space $K$, 
which have certain behaviors on fixed open or closed subsets $E$ of $K$.  Usually 
$E$ (or its complement) is a peak set, or an intersection of peak sets.  
In \cite{Hayth,Hay,BHN} Hay, the first author, and Neal, began a program of
 noncommutative peak interpolation, with closed projections in $A^{\perp \perp}$
playing the role of peak sets.  A very nice  application 
of these ideas appears in \cite{Ueda}.
One may find several peak interpolation results in \cite{Hayth,Hay,BHN}.
For example, our noncommutative  Urysohn lemmas in 
\cite{BRead,BNII} are noncommutative peak interpolation theorems. 
  Another
example: in  \cite[Proposition 3.2]{Hay} 
it is shown that if $A$ is a unital operator algebra, if $q$ is a closed projection
in $A^{\perp\perp}$,  and if we are given $b \in A$  
with $\Vert b q \Vert \leq 1$, and an $\epsilon > 0$, then
there exists $a
\in (1+\epsilon) {\rm Ball}(A)$ such that $a q =bq$.   In the commutative case
one may take $\epsilon = 0$ here, but  it has been open for several years whether this 
is true for  noncommutative algebras.
In this section we answer this question.   We had noticed earlier that the question was 
related to an open question that dates to the time of \cite{BEZ} about proximinality (discussed
in the intriduction) which we settle too.  

\begin{theorem} \label{peakth}   Suppose that $A$ is an  operator algebra 
(not necessarily approximately unital), $p$ is an open projection in $A^{**}$,
and $b \in A$ with $bp= p b$.   Then $b$ achieves its
distance to the $\ell$-ideal $J = \{ a \in A : ap = a \}$  associated with $p$  (that is, there exists 
 a point $x \in J$ with $\Vert b - x \Vert =
d(b,J)$).  \end{theorem}   \begin{proof}    It is well known, and straightforward,
 that for an $\ell$-ideal $J$ 
with support projection $p$ in a unital operator algebra $A$, if $b \in A$ then
$d(b,J) = \Vert b (1-p) \Vert$
(indeed in the second dual of $A/J$, which may be identified with $A^{**}/J^{\perp \perp} =
A^{**}/A^{**}p \cong A^{**} (1-p)$, the canonical copy of $b + J$ corresponds to $b (1-p)$).
 Let $D = pA^{**} p \cap A$, the HSA associated with $p$. 
We have  $$b D = b p D =
pb D \subset pA^{**} p \cap A = D,$$ and similarly $D b \subset D$. 
Now $d(b,J) = \Vert b (1-p) \Vert$ by the fact at the start of the proof  applied in the unitization  $A^1$. 
If $C$ is the closed unital algebra generated by
$b$  and $D$ and the identity of $A^1$, then $D$ is an approximately unital ideal in $C$,
and $p$ is its support projection in $C^{**} \cong C^{\perp \perp} \subset (A^1)^{**}$.
However approximately unital ideals are $M$-ideals
(an observation of Effros and Ruan, see 
e.g.\ \cite[Theorem 4.8.5 (1)]{BLM}),
and hence are proximinal \cite{HWW,BAIC}.  Thus (and also applying
the fact at the start of the proof  again but now with respect to the ideal $D$ of $C$), there exists 
an element $d \in D \subset J$ with $\Vert b - d \Vert = d(b,D) = \Vert b (1-p) \Vert = d(b,J) .$
  \end{proof}

Note that the last result applies to every $\ell$- or $r$-ideal in $A$ (and to all $b \in A$ commuting with the support projection of the
ideal).
 
The following is an `$\epsilon = 0$ variant' of the peak  interpolation result mentioned above Theorem \ref{peakth}.  It is
a generalization of a classical fact about peak sets \cite{Gam} mentioned in our introduction.

\begin{corollary} \label{peakth2}  
 Suppose that $A$ is an operator algebra (not necessarily approximately unital),
 $p$ is an open projection in $A^{**}$,
and $b \in A$ with $bp= p b$ 
 and $\Vert b (1-p) \Vert \leq 1$ (where $1$ is the identity of the unitization of $A$ if $A$ is nonunital).
Then  there exists an element $g \in {\rm Ball}(A)$ with $g (1-p) =  (1-p) g = b(1-p)$.  
\end{corollary} 

\begin{proof}    Suppose that $x$ is a best approximation to $b$ found in the 
previous proof in the $\ell$-ideal $J$ supported by $p$,
and let $g = b - x$.  Then $g (1-p) = b(1-p)$ since $x p = p x =  x$ (the latter 
since $x \in D$ in the proof of Theorem \ref{peakth}.    If $A$ is nonunital then $J$ is
also the $\ell$-ideal in $A^1$  supported by $p$.  Indeed 
$A^{**} p \cap A = (A^1)^{**} p \cap A^1$  since if $x \in A^1$ with $x = xp \in A^{\perp \perp}$,
then $x \in A^1 \cap A^{\perp \perp} = A$.   As stated in  the previous proof,
 we have $\Vert b (1-p) \Vert = d(b,J)$, which equals 
$\Vert g \Vert$.  So  $g \in {\rm Ball}(A)$.
  \end{proof}

It is of interest to replace
the $1-p$ in the last result by a `compact' projection $q$ in $A^{**}$.
This is possible, as we shall discuss in Section 5.

We now turn to showing that the results above are best possible.
 
\begin{theorem} \label{nomi}   Not every  left ideal with a cai  in  a unital operator algebra is proximinal.  
\end{theorem} 

\begin{proof}  Identify $c_0$ as the `main  diagonal' of $\Kdb(\ell^2)$, and $C$ as the `first column'
 of $\Kdb(\ell^2)$, and let $A = 
C + c_0 + \Cdb I$, a closed unital subalgebra of $B(\ell^2)$.  Explicitly,
$A$ consists of infinite matrices
$$
a=\left( \begin{array}{ccccccc}
v_1 & 0 & 0 & 0 & 0 & 0 & \ldots \\
v_2& d_2 & 0 & 0 & 0 & 0 & \ldots \\
v_3 & 0 & d_3 & 0 & 0 & 0 & \ldots \\
v_4 & 0 & 0 & d_4  & 0 & 0 & \ldots \\
v_5 & 0 & 0 & 0 & d_5  & 0  & \ldots \\
\ldots& \ldots&\ldots&\ldots&\ldots&\ldots&\ldots\end{array} \right)
$$
with $\vec v = (v_j) \in \ell^2$, and $(d_j) \in c$, the space of all convergent sequences.
Let $J$ be the copy of
$c_0$ multiplied by $I - E_{11}$, which is
an $\ell$-ideal in $A$, involving those matrices $a$ as above with $\vec v = \vec 0$ and 
$\lim_n d_n = 0$.  Set $b \in A$ to be a matrix like $a$ as above 
but with $1$'s on the main diagonal (so $1 = v_1 = d_2 = d_3 = \cdots$), and where $v_j \neq 0$ for all $j$.
    Since $J E_{11} = (0)$ and 
$$\Vert \vec v \Vert 
= \Vert b E_{11} \Vert =  \Vert (b - y)  E_{11} \Vert \leq 
\Vert b - y \Vert , \qquad  y \in J,$$   we have $d(b,J) \geq \Vert \vec v \Vert
=  \Vert \vec v \Vert_2$.
If $e_n = E_{22} + \cdots + E_{nn} \in J$  then one may write  $b - e_n$  as the sum of $\vec v_n \oplus I$ and
$\vec v - \vec v_n$, where $\vec v_n = (E_{11} + e_n) \vec v$.  Hence $$\Vert  b - e_n \Vert 
\leq \max \{ \Vert \vec v_n \Vert , 1 \} + \Vert \vec v -  \vec v_n \Vert =   \Vert \vec v_n \Vert_2 + 
\Vert \vec v -  \vec v_n \Vert_2 \to \Vert \vec v \Vert_2.$$
We deduce that $d(b,J) =  \Vert \vec v \Vert_2.$   We now show that this distance is not achieved.
By way of contradiction suppose that  $y \in J$ with $\Vert  b-y \Vert = d(b,J)$.
Now $b -y$ has first column $\vec v$, and for $j > 1$ the $j$th column $\vec c_j$  of  $b -y$ has at most one nonzero entry, and that entry is $1 - y_j$, where $(y_j) \in c_0$.  If  
 $c$ is the matrix with two columns  $\vec v$ and $\vec c_j$, then 
$$\Vert \vec v  \Vert  \leq   \Vert  c  \Vert  \leq \Vert  b-y \Vert .$$
Hence if $\Vert  b-y \Vert = d(b,J) = \Vert \vec v  \Vert_2$ then $\Vert  c^* c  \Vert  =
\Vert  c  \Vert^2 =  \Vert \vec v  \Vert_2^2$.
However the first row of $c^* c$ is $( \Vert \vec v  \Vert_2^2 , \bar{v_j} (1-y_j))$,
and so   $\bar{v_j} (1-y_j) $ is forced to be $0$, and hence $y_j = 1$ for all $j$.
But this is impossible since $(y_j) \in c_0$.    
 \end{proof}

\begin{corollary} \label{nopi}   Suppose that $A$ is a  unital operator algebra  and that $q$ is a peak projection
for $A$.
If  $b \in A$ with $\Vert bq \Vert \leq 1$ then there need not exist an element $g \in {\rm Ball}(A)$ with $g q = bq$.  
\end{corollary} 

\begin{proof}    Let $A, J,$ and $b$ be as in the proof of Theorem \ref{nomi}, with $d(b,J)$ 
not achieved.  Since $A$ is separable, the complement of 
the support projection of $J$ is a peak projection $q$
by \cite[Section 2]{BRead} (see also \cite[Theorem 3.4 (2)]{BNII}).  Let $b' = b/\Vert bq \Vert$.  If there was an element $g \in {\rm Ball}(A)$ with $g q = b'q$
then a slight modification  to the  proof of \cite[Proposition 6.6]{BHN} shows that  $d(b,J)$ is achieved. 
 \end{proof}

The last corollary shows that there is probably little point in looking for more sophisticated and completely 
general  peak interpolation
results of this flavor, other  than the ones we have 
already found.   Clearly the way to proceed from this point, in noncommutative peak interpolation,  is to insist on a 
commutativity assumption of the type considered earlier in this section.  

\medskip

{\bf Remark.}  We can answer a question raised in \cite[Section 6]{BHN}. It follows from the above and  \cite[Proposition 6.6]{BHN} that not every $p$-projection in the sense of \cite{Hay,BHN}  is a
`strict $p$-projection'
as defined above  \cite[Proposition 6.6]{BHN}.  

Since r-ideals and $\ell$-ideals are examples of the (complete) one-sided $M$-ideals of \cite{BEZ}
(see Proposition 6.4 there), we can also answer a question raised around the time of that 
investigation (see e.g.\ \cite[Chapter 8]{BZ}):    

\begin{corollary} \label{nomid}    One-sided $M$-ideals in operator spaces
in the sense of {\rm \cite{BEZ}} need not be proximinal.  
\end{corollary} 

Note that $A$ in our example above is very simple as an operator space, indeed it is a subalgebra of the nuclear locally reflexive $C^*$-algebra 
$\Kdb(\ell^2) + \Cdb I$.   It is clearly exact and  locally reflexive (since these properties are hereditary
\cite{P}).
Thus we see that the obstacle to  proximinality,
and to more sophisticated peak interpolation results
than the ones we have already obtained, is not an operator space
phenomenon, rather it is simply that one needs a certain amount of commutativity.

\section{Elements with positive real part}

If $A$ is any nonunital operator algebra then as we said earlier $A^1$ is uniquely defined, and hence so is
$A^1 + (A^1)^*$ by e.g.\ 1.3.7 in \cite{BLM}.  We define $A + A^*$ to be the obvious subspace
of $A^1 + (A^1)^*$.  This is well defined.  To see this, suppose that $A$ is a subalgebra of
$B(H)$, and that $\theta : A \to B(K)$ is a completely isometric (resp.\ 
completely contractive) homomorphism.
By Meyer's result (see \cite[Section 2.1]{BLM}, the map  $\lambda I_H + a \mapsto \lambda I_K + \theta(a)$ is
completely isometric (resp.\ 
completely contractive), and by e.g.\ 1.3.7 in \cite{BLM} it extends further to a 
unital completely 
isometric complete order 
isomorphism (resp.\ 
completely contractive unital) from $\Cdb  I_H + \, A + A^* \to \Cdb I_K + \, \theta(A) +  \theta(A)^*$, 
namely $\lambda I_H + a + b^* \mapsto \lambda I_K + \theta(a) + \theta(b)^*$.  
This last map restricts to a $*$-linear  completely isometric (resp.\  
completely contractive) surjection $A + A^* \to  \theta(A) +  \theta(A)^*
: a + b^* \mapsto \theta(a) + \theta(b)^*$, for $a, b \in A$.  Thus a statement such as
$a + b^* \geq 0$ makes sense whenever $a, b \in A$, and is independent of the particular $H$ on which $A$
is represented.  We set ${\mathfrak r}_A = \{ a \in A : a + a^* \geq 0 \}$.   This is a closed cone in $A$, and is
weak* closed if $A$ is a dual operator algebra.  

If $x \in A$ with $x + x^* \geq 0$
then  $x$ has a unique $m$th root $x^{\frac{1}{m}}$ for each $m \in \Ndb$
with numerical range having argument in $(-\frac{\pi}{m}, \frac{\pi}{m})$ (see \cite[Theorem 0.1]{LRS}).    It is
shown there that $x^{\frac{1}{m}} \in {\rm oa}(x)^1$, but 
in fact obviously this root lies in ${\rm oa}(x)$ if the latter is
nonunital (since if $x^{\frac{1}{m}} = \lambda 1 + a$ for $\lambda \in \Cdb$ and 
$a \in {\rm oa}(x)$ then by taking $m$th powers 
$x - \lambda^m 1 \in {\rm oa}(x)$).   
We thank the referee for providing a proof of the  
next result, which may be known to experts on sectorial
operators.

\begin{theorem}  \label{sect}  If $x \in B(H)$ with $x + x^* \geq 0$
then $x^{\frac{1}{m}} \,  x \to x$ and $\Vert x^{\frac{1}{m}} \Vert \to 1$ as $m \to \infty$.
Thus a normalization of $(x^{\frac{1}{m}})$
is a cai for oa$(x)$. \end{theorem}

\begin{proof}  We use the machinery in \cite[Theorem 1.2]{LRS}.  In particular,
 let $\theta$ be a number slightly bigger than
$\frac{\pi}{2}$,
and let $\Gamma$ be the positively oriented  closed curve in the plane
 with three pieces: the two line segments $\Gamma_1$ and $\Gamma_3$
from $0$ to the two points $R e^{i \theta}$ and $R e^{-i \theta}$, and $\Gamma_2$ the right part of the
circle radius $R$ centered at the origin connecting to the latter two points.   We choose $R$ so that
the numerical range of $x$ is inside the circle radius $R$.
Define $\Gamma_1(\epsilon)$ to be $\Gamma_1$ with the last part of it,
a segment of length $\epsilon$, removed.
Similarly define $\Gamma_3(\epsilon)$ to be the part of $\Gamma_3$ bounded away from $0$.  Let 
$\Gamma(\epsilon)$ be the part of $\Gamma$ comprised by $\Gamma_3(\epsilon), 
\Gamma_2$ and  $\Gamma_1(\epsilon)$; this is
a connected but not closed curve, and 
 \cite{LRS} defines 
$x^t = \lim_{\epsilon \to 0^+} \, \frac{1}{2 \pi i} \int_{\Gamma_\epsilon} \,
\lambda^{t} (\lambda 1 - x)^{-1} \, d \lambda$, for $t \in (0,1)$,
showing that this limit exists in norm.
  Next consider the integral $\frac{1}{2 \pi i} \int_{\Gamma} \,
(\lambda^{1+t} - \lambda) (\lambda 1 - x)^{-1} \, d \lambda$.   It is known that since
$x$ is `accretive', we have that $\lambda (\lambda 1 - x)^{-1}$ is bounded on
$(\Gamma_1 \cup \Gamma_3) \setminus \{ 0 \}$ 
(see e.g.\ \cite[Lemma C.7.2 (v)]{Haase}), and by continuity it is also bounded on $\Gamma_2$.
Thus there is a constant $K$ with
\begin{equation} \label{eqnio}  \Vert \lambda 
(\lambda 1 - x)^{-1} \Vert \leq K , \qquad \lambda \in \Gamma \setminus \{ 0 \} .  \end{equation}
It follows that
$$\Vert \int_{\Gamma_1(\epsilon)}  \, \lambda^{t} (\lambda 1 - x)^{-1} \, d \lambda \Vert
\leq K \int_{\Gamma_1(\epsilon)}  \, |\lambda^{t-1}| ds = K \int_{\epsilon}^R \, r^{t-1} \, dr
= K(\frac{R^t}{t} - \frac{\epsilon^{t}}{t}) .$$
A similar bound holds for
 $\Gamma_3(\epsilon)$.
On $\Gamma_2$ the continuous function $(\lambda 1 - x)^{-1}$ is bounded, and since
$|\lambda^{t}| = R^t \leq \max \{ 1, R \}$ here,
 we see that  $\int_{\Gamma_2} \,
\lambda^{t} (\lambda 1 - x)^{-1} \, d \lambda$ is bounded
independently of $t$.  Letting $\epsilon \to 0$ we
deduce that $\Vert x^t \Vert \leq \frac{C}{t}$ for some constant $C > 0$.

We now consider contour integrals over $\Gamma$.  Since $0 \in \Gamma$
this is not the usual Riesz functional calculus, but rather
an extended version of it of the type considered e.g.\ in \cite{Haase}.
By the functional calculus for such contours, we have
$$x^{1+t} - x = \frac{1}{2 \pi i} \int_{\Gamma} \, (\lambda^{1+t} - \lambda) (\lambda 1 - x)^{-1} \, d \lambda
= \frac{1}{2 \pi i} \int_{\Gamma} \, (\lambda^t - 1) \lambda (\lambda 1 - x)^{-1} \, d \lambda .$$
We are using here the fact that $\Vert \int_{\Gamma \setminus \Gamma(\epsilon)}
(\lambda^t - 1) \lambda (\lambda 1 - x)^{-1} \, d \lambda
\Vert$ is dominated, by Equation (\ref{eqnio}),  by a constant times the
length of $\Gamma \setminus \Gamma(\epsilon)$,  which goes to $0$.
Indeed by the same argument, there is a constant $D$ with
 $\Vert x^{1+t} - x \Vert \leq D \int_{\Gamma} \, |\lambda^t - 1| ds.$
By Lebesgues dominated convergence theorem we
deduce that $x^{\frac{1}{m}} x \to x$ as $m \to \infty$.

 As we said earlier,  $x^{\frac{1}{m}}$ has  numerical range having argument in $(-\frac{\pi}{m},
\frac{\pi}{m})$ (see \cite[Theorem 0.1]{LRS}).
Write $x^{\frac{1}{m}} = a_m + i b_m$ for selfadjoint
$a_m, b_m$ with $a_m \geq 0$.  Suppose that $k < m$.
Expand $x^{\frac{k}{m}} =  (a_m + i b_m)^k$, and use this
to write $a_m^k$ as $x^{\frac{k}{m}}$ minus various products  of powers
of $a_m$ and $b_m$.   Applying the triangle inequality to the
the latter, one obtains $$\Vert a_m \Vert^k \leq \Vert x^{\frac{k}{m}} \Vert
 + \sum_{j=1}^k
\, {k \choose j} \, \Vert b_m \Vert^j  \Vert a_m \Vert^{k-j} .$$
It follows that
$$\frac{C m}{k} \geq \Vert x^{\frac{k}{m}} \Vert \geq
\Vert a_m \Vert^k -
\sum_{j=1}^k  \, {k \choose j} \, \Vert b_m \Vert^j \Vert \Vert a_m \Vert^{k-j}
= \Vert a_m \Vert^k (2 - (1+ \frac{\Vert b_m \Vert}{\Vert a_m \Vert})^k).$$
By the numerical range fact above,
 there is a constant $c > 0$
with $$|\psi(b_m)| < \frac{c}{m} |\psi(a_m)| \leq \frac{c}{m} \Vert a_m \Vert, $$ for every state $\psi$.  Thus
$\Vert b_m \Vert \leq \frac{c}{m} \Vert a_m \Vert$.  We deduce that
$\Vert a_m \Vert^k (2 - (1+\frac{c}{m})^k) \leq \frac{C m}{k}$.
If we ensure that $k \leq \beta m$ for  a fixed $\beta$ with
 $0 < \beta < \frac{\ln \frac{3}{2}}{c}$, then for $m$ and $k$ large
enough we have $2 - (1+\frac{c}{m})^k > \frac{1}{2}$, and so 
$\Vert a_m \Vert < 
(\frac{2C m}{k})^{\frac{1}{k}}$.  If $k$ dominates a small constant times $m$,
it follows that $\limsup_m \Vert a_m \Vert \leq 1$.  Hence
$\limsup_m \Vert x^{\frac{1}{m}}  \Vert \leq 1$.
On the other hand,
if a subsequence $\Vert x^{\frac{1}{m_k}}  \Vert \to \beta < 1$, then
this would  contradict the
fact that $x^{\frac{1}{m}} x \to x$.  So $\lim_m \Vert x^{\frac{1}{m}} \Vert \to 1$.   \end{proof}

{\bf Remark.}  The fact that if $x + x^* \geq 0$  then
 ${\rm oa}(x)$ has a cai was deduced in an earlier version of the paper
from the next result and \cite[Lemma 2.1]{BRead}.
 
\begin{theorem}  \label{oax}  If $A$ is an operator algebra 
and $x \in A$ with $x + x^* \geq 0$
then  ${\rm oa}(x) = {\rm oa}(y)$
for some $y \in \frac{1}{2} {\mathfrak F}_A$.  One 
may take $y = x (I+x)^{-1}$, and if so 
$\Vert y \Vert \leq \frac{\Vert x \Vert}{\sqrt{1 + \Vert x \Vert^2}}$. 
\end{theorem}

\begin{proof} 
Suppose that $A$ acts on a Hilbert
space $H$ and write $I$ for $I_H$, and 
identify oa$(x)^1$ with oa$(I,x)$, the operator algebra generated by
oa$(x)$ and $I$ in $B(H)$.   The numerical range of $x$ in the 
latter algebra lies in the right hand half plane.
We have  $-1 \notin 
{\rm Sp}_{{\rm oa}(I,x)}(x)$ and so $I+x$ is 
invertible in ${\rm oa}(I,x)$.
Set $y = x (I+x)^{-1} \in {\rm oa}(x)$.    
Then $I - 2y = (I-x) (I+x)^{-1}$, the Cayley transform of 
$x$, which is well known (see e.g.\ 
\cite[IV, Section 4]{Sz}), and is easy to see,  is a contraction
if $x + x^* \geq 0$.   Hence $y \in \frac{1}{2} {\mathfrak F}_A$.
Note that $I - y = (I+x)^{-1}$.  It follows from 
Lemma \ref{abr} that oa$(I-y)$ contains  $(I-y)^{-1} =  I + x$
and  $I$.
So  $(I-y)^{-1} \in {\rm oa}(I-y) \subset {\rm oa}(I,y) 
\subset {\rm oa}(I,I-y) = {\rm oa}(I-y)$.  Hence  $x = y (I-y)^{-1} \in {\rm oa}(y)$.   
Thus ${\rm oa}(x) = {\rm oa}(y)$.  

Representing $A \subset B(H)$, for $\zeta \in H$ we have 
$$\Vert (1+x) \zeta \Vert^2 = \langle (1 + x + x^* + x^* x) \zeta , \zeta 
\rangle \geq (\frac{1}{\Vert x \Vert^2} + 1)  \langle x^* x  \zeta , \zeta \rangle  = 
(1 + \frac{1}{\Vert x \Vert^2})  \Vert  x  \zeta  \Vert^2 .$$
Thus $$\Vert x (I+x)^{-1} \zeta \Vert \leq  \frac{\Vert x \Vert}{\sqrt{1 + \Vert x \Vert^2}} \Vert \zeta \Vert
\leq \frac{\Vert x \Vert}{\sqrt{1 + \Vert x \Vert^2}} ,$$
for $\zeta \in {\rm Ball}(H).$  So 
$\Vert y \Vert \leq \frac{\Vert x \Vert}{\sqrt{1 + \Vert x \Vert^2}}$.  \end{proof}

The following result was found by the first author, during a discussion with S. Sharma.
The present proof however is an observation of the referee.

\begin{theorem}  \label{cls}  If $A$ is any (not necessarily
 approximately unital)  operator algebra then the closure of
 $\Rdb^+ {\mathfrak F}_A$ equals $\{ x \in A : x + x^* \geq 0 \}$.
\end{theorem}

\begin{proof}    
If $x \in {\mathfrak F}_A$ then $x + x^* \geq x^* x \geq 0$.  It follows that
  the closure of
 $\Rdb^+ {\mathfrak F}_A$ is contained in the closed cone $\{ x \in A : x + x^* \geq 0 \}$.  For the converse, 
if $x + x^* \geq 0$ then we appeal to the proof of the last result.  We have
$x = \lim_{t \to 0^+} \, \frac{1}{t} t x (1 + tx)^{-1}$, and $t x (1 + tx)^{-1} \in {\mathfrak F}_A$.  
\end{proof}

We are currently  working on implications of some of the results above with S. Sharma.
We mention a few now.  Following the ideas in the route presented in \cite{BRead} 
one immediately obtains:

\begin{corollary}   \label{supp}  {\rm (Cf.\ \cite[Lemma 2.5]{BRead})} \ For any operator algebra $A$,
if $x \in A$ with $x + x^* \geq 0$ and  $x \neq 0$,
then the left support projection of $x$
 equals the right support projection, and equals $s(x (1+x)^{-1})$,
where $s(\cdot)$ is the support projection 
studied in {\rm \cite{BRead}}.
If $A \subset B(H)$ via a representation $\pi$, for a Hilbert space
$H$, such that the unique weak* continuous extension $\tilde{\pi} :
A^{**} \to B(H)$ is (completely) isometric, then this
support projection $s(x)$ also may be
identified with the smallest projection $p$ on $H$ such that $p x =
x$ (and $x p = x$). That is, $s(x)H  = \overline{{\rm Ran}(x)} =
{\rm Ker}(x)^\perp$. Also,  $s(x)$ is an open projection in $A^{**}$
in the sense of  {\rm \cite{BHN}}. 
 \end{corollary}

We write $s(x)$ for the support projection in the last result
for $x \in {\mathfrak r}_A$.

\begin{corollary}  \label{supp2}  {\rm (Cf.\ \cite[Corollary 2.6]{BRead})} \ For any operator algebra $A$, if $x \in A$ with $x + x^* \geq 0$ then
 the closure of $xA$ is an r-ideal in $A$ and $s(x)$ is the support projection
of this r-ideal.  We have $\overline{xA} = 
\overline{yA} = s(x) A^{**} \cap A$, where 
$y \in \frac{1}{2} {\mathfrak F}_A$ is as in Theorem  {\rm \ref{oax}}.  The analogous results hold for 
$\overline{Ax}$, and this  is the $\ell$-ideal matching $\overline{xA}$.   Also,
$\overline{xAx}$ is the HSA matching $\overline{xA}$.
\end{corollary}

Thus our descriptions of r-ideals and $\ell$-ideals and HSA's from \cite{BRead} 
in terms of ${\mathfrak F}_A$, may be rephrased in terms
of the $x \in A$ with $x + x^* \geq 0$.  
 Corollaries 2.7 and 2.8 of \cite{BRead} are 
true with $x \in {\mathfrak F}_A$ replaced by $ {\mathfrak r}_A$.  Most of Lemma 2.10 of \cite{BRead} also generalizes 
to this case, with the exception of (iv) and (v).
Theorem 3.2  of \cite{BRead} generalizes
to this case too. 
We present the easy details of the proofs elsewhere.

Similarly, all results in \cite[Section 8]{BRead} generalize.
For example, if one  defines a map $T : A \to B$  to be
{\em real completely positive} if $T(x) + T(x)^* \geq 0$ whenever  $x \in A$ with $x + x^* \geq 0$
(and the obvious matching matricial version of 
this assertion holding for  $x \in M_n(A)$, for all $n \in \Ndb$),
then  a map  on an approximately unital operator algebra or operator system
is real completely positive iff
it is OCP in the sense of \cite{BRead}.   The following is the analogue of
\cite[Lemma 8.1]{BRead}.
 
\begin{corollary}  \label{supp3}   Suppose that $A$ is an approximately unital operator algebra.
Then $\{ x \in A : x + x^* \geq 0 \}$ is 
weak* dense in  $\{ x \in A^{**} : x + x^* \geq 0 \}$. 
\end{corollary}

\begin{proof}   We stated earlier that  $\{ x \in A^{**} : x + x^* \geq 0 \}$ is weak* closed.
If $\eta$ is in the latter set, it is a limit of elements of the form $t y$ for $t \geq 0$ and 
$y \in   {\mathfrak F}_{A^{**}}$, 
by Theorem \ref{cls}.  However $ty$ is a weak* limit 
of elements in $t  {\mathfrak F}_A \subset \{ x \in A : x + x^* \geq 0 \}$ by  \cite[Lemma 8.1]{BRead}.
\end{proof}

Many results from e.g.\ \cite{BNI, BNII} generalize too (i.e.\
using $\{ x \in A : x + x^* \geq 0 \}$ in place of ${\mathfrak F}_{A}$), by virtue of
 e.g.\ the rules for powers given in \cite[Lemma 1.1]{BNI},
which are still valid by an obvious proof using \cite[Corollary 1.3]{LRS}.

\section{Nonunital operator algebras}

If $A$ is a unital or approximately unital operator algebra then ${\mathfrak F}_A$ seems quite manageable.
Hitherto we had assumed that ${\mathfrak F}_A$ could be badly behaved if $A$ was not approximately
  unital.  However we shall 
see below that  in this case ${\mathfrak F}_A =  {\mathfrak F}_C$ for an approximately unital subalgebra $C$ (which might be $(0)$).

If $A$ is any operator algebra, define $A_H = \overline{{\mathfrak F}_A \, A {\mathfrak F}_A}$.
This will play an important role in the sequel.  
We define $A_r = \overline{{\mathfrak F}_A \, A}$
and $A_\ell = \overline{A \, {\mathfrak F}_A}$.   By e.g.\ \cite[Corollary 2.6]{BRead}
and the fact from the introduction
that the closure of a sum of r-ideals is an r-ideal,
$A_r$ is an r-ideal.  Similarly, $A_\ell$ is an $\ell$-ideal.  In 
fact these are the largest r-ideal and $\ell$-ideal in $A$ 
(as may be seen using \cite[Theorem 2.15]{BRead}).    By the proof in the introduction 
that the closure of a sum of r-ideals is an r-ideal, the support projection of 
$A_r$ is $p = \vee_{x \in  {\mathfrak F}_A} \, s(x)$, where $s(x)$ denotes the  support  projection of 
$x$ (see \cite[Section 2]{BRead}), since $s(x)$ is
the support projection of $\overline{xA}$ (by e.g.\ Corollary \ref{supp2}). 
 and joins).  Similarly, the support projection of 
$A_\ell$ is $p$, and now we see that $A_\ell$ is the $\ell$-ideal associated with $A_r$ (see \cite[Section 2]{BHN}).
We also see that $A_H = \overline{{\mathfrak F}_A \, A {\mathfrak F}_A} = A_r A_\ell$ is the 
HSA associated with this r-ideal  (see \cite[Section 2]{BHN}).

\begin{proposition} \label{chau}  An operator algebra $A$ has a cai
iff the  span of ${\mathfrak F}_A$ is dense in $A$.
\end{proposition} \begin{proof}  ($\Rightarrow$) \ If $A$ is unital then
it is the span of ${\mathfrak F}_A = 1 + {\rm Ball}(A)$ obviously. 
Thus in the general case $A^{**}$ is the  
span of ${\mathfrak F}_{A^{**}}$.  
If $\varphi \in ({\mathfrak F}_A)^{\perp}$ then
$\varphi$ annihilates the span of the weak* closure 
of ${\mathfrak F}_A$.  This weak* closure is ${\mathfrak F}_{A^{**}}$
by \cite[Lemma 8.1]{BRead}, and so $\varphi = 0$ on $A^{**}$ and so is zero.  Thus Span$({\mathfrak F}_A)$ is dense.

($\Leftarrow$) \ If the span of ${\mathfrak F}_A$ is dense in $A$
then $A_r = A_\ell = A$, using the existence of roots of elements of 
${\mathfrak F}_A$.  Hence $A$ has a right cai and a left cai,
and therefore has a cai (by e.g.\ \cite[Proposition 2.5.8]{BLM}). 
 \end{proof}  

\begin{theorem}  \label{sps}  If $A$ is any operator algebra then the closure of 
the linear  span of ${\mathfrak F}_A$
 is a HSA in $A$.  Indeed it is the biggest 
approximately unital operator algebra inside $A$, and equals $A_H$.  Moreover
${\mathfrak F}_A = {\mathfrak F}_{A_H}$.
\end{theorem} 

\begin{proof}   Let $D = \overline{\spn}({\mathfrak F}_A)$ and $C = A_H$.   Clearly  ${\mathfrak F}_C \subset 
 {\mathfrak F}_A$.  Conversely, since any $x \in  {\mathfrak F}_A$ has a third root,  we have
$x \in C$, so that $x \in  {\mathfrak F}_C$.  Thus ${\mathfrak F}_C = {\mathfrak F}_A$. 
 Hence $D =  \overline{{\rm Span}}({\mathfrak F}_C) =  C$ by Proposition \ref{chau}, and it
is clearly a HSA.    If $B$ is an approximately unital subalgebra
of $A$ then  ${\mathfrak F}_B \subset  {\mathfrak F}_A = {\mathfrak F}_C$, and so $B \subset C$
by  Proposition \ref{chau}.  
\end{proof} 

\begin{corollary}   \label{ah}  
Let  $A$ be any operator algebra.  \begin{itemize} 
 \item [(1)]    $A_H = \overline{{\rm Span}}({\mathfrak r}_A)$, and ${\mathfrak r}_A = {\mathfrak r}_{A_H}
\subset A_H$ .  
\item [(2)]   $A$ has a cai iff $A = \overline{{\rm Span}}({\mathfrak r}_A)$.
\end{itemize} 
 \end{corollary}

\begin{proof} 
(1) \ The first assertion is obvious from Theorem \ref{cls} and Theorem  \ref{sps}.   So 
${\mathfrak r}_A \subset A_H$, and the second is now obvious.  

(2) \ Follows from Theorem \ref{cls} and Proposition \ref{chau}.
 \end{proof}

Thus a finite dimensional operator algebra has an identity of norm $1$ 
iff it is spanned by ${\mathfrak r}_A$, and it contains an
orthogonal projection iff  
${\mathfrak r}_A \neq (0)$.  For the latter, note that if ${\mathfrak r}_A \neq (0)$ then 
$A_H$ is a nontrivial unital algebra, and so its identity is a projection.
  
We recall that $A^2$ is the {\em closure} of the span of products of two elements from $A$.

\begin{corollary}  \label{jl}  If $A$ is an operator algebra such that  $A^2$ has a cai,
then $A_r = A_\ell = A_H =\overline{\spn}({\mathfrak F}_A)  = \overline{{\rm Span}}({\mathfrak r}_A) = A^2$.
\end{corollary} 
 
\begin{proof}  Note that $A_r$ and $A_\ell$ are subsets of $A^2$,
and $(A^2)_r \subset A_r$.  If $A^2$ has a cai then by
the proof of Proposition \ref{chau} we have $A^2 = (A^2)_r \subset A_r \subset A^2$.
So $A_r = A^2$ and similarly $A_\ell  = A^2$.   The rest is clear from 
Theorem \ref{sps}, Corollary   \ref{ah} (1),  and the definition of $A_H$
(giving $A^2 \subset A_H \subset A_r$).  \end{proof}  

{\bf Remarks.} 1) \  If $A_r = A_\ell$ then this is the largest closed ideal with cai in $A$, since if $J$ is any
closed ideal with cai in $A$ then $J = J_r \subset A_r = A_\ell$.

2) \ A similar result to Corollary  \ref{jl} holds with $A^2$ replaced by $A^n$ for any $n \in \Ndb$, with
an almost identical proof.  

\medskip

One may use the above to extend 
much of the theory of operator algebras with cai, to arbitrary 
operator algebras.
 For example, there exist nice relationships between the states of a nonunital operator 
algebra $A$ (defined as the nonzero functionals
that extend to states on $A^1$) and quasistates on $A_H$  (we recall that a
quasistate is a state multiplied by a scalar in $[0,1]$).   If $\varphi$ is
a state on $A^1$ then $\varphi_{\vert A_H}$ is a  quasistate on $A_H$.  Indeed, $\varphi$
extends further to a state on $C^*(A^1)$, and if $p$ is the support 
projection of $A_H$ then $0 \leq \varphi(p) = \lim_t \, \varphi(e_t)$  for some 
cai for $A_H$, hence $\varphi(p) \leq \Vert \varphi_{\vert A_H} \Vert$.  Conversely, 
by the Cauchy-Schwarz inequality, if $x \in {\rm Ball}(A_H)$ then 
$|\varphi(x)| = |\varphi(pxp)|  \leq \varphi(p)$, so that $\varphi(p) \geq \Vert \varphi_{\vert A_H} \Vert$.
It follows by e.g.\ 2.1.18 and 2.1.19 in \cite{BLM} that  $\varphi_{\vert A_H}$ is a  quasistate on $A_H$. 
Note too that by \cite[Theorem 2.10]{BHN} a functional
$\psi$  on $A_H$ has a {\em unique} Hahn-Banach extension $\tilde{\psi}$ on
 $A^1$.  The latter will be a state if $\psi$ is, by the last line of 
2.1.19 in \cite{BLM} (taking $A_H^1 = A_H + \Cdb 1_{A^1}$, so that 
 $\tilde{\psi}(1) = 1 = \Vert \tilde{\psi} \Vert$).

From this it follows that 
results such as Lemma 2.9 in \cite{BRead} are true for 
nonunital operator algebras too.
However many results do not
carry over.  For example ${\mathfrak  F}_A$ need not be weak* dense in
   ${\mathfrak  F}_{A^{**}}$ if $A$ is nonunital (cf.\ \cite[Lemma 8.1]{BRead}).  A counterexample is given 
by $A$ equal to the functions in the disk 
algebra vanishing at $0$.
Here ${\mathfrak  F}_A = (0)$, but ${\mathfrak  F}_{A^{**}} \neq (0)$ since $A^{**}$ has many 
projections.  Indeed for example the function $f = z (z+1)/2 \in A$ peaks at $1$, and  
by Lebesgue's theorem $(f^n)$ converges weak* to the canonical 
copy $q$ in $C(\Tdb)^{**}$ of the characteristic function of $\{ 1 \}$,  and $q^2 = q$
and $\Vert q \Vert \leq 1$, so $q$ is
a nontrivial projection in $A^{**}$.  That the bidual has many projections can 
also be seen since the bidual of its unitization, the disk algebra, does;
 and if $A$ is any nonunital Arens regular Banach algebra, and $A^1$ is a
unitization of $A$, then as soon as $(A^1)^{**}$ has nontrivial projections, then so does
$A^{**}$.  To see this consider the continuous homomorphism  
$\chi : A^1 \to \Cdb : a + c 1 \mapsto c$, whose 
canonical weak* extension  $\tilde{\chi} : (A^1)^{**}  \to \Cdb$ has kernel  $A^{\perp \perp}$.
If $p$ is a nontrivial projection in $(A^1)^{**}$ then we have $\tilde{\chi}(p) = 1$ or $0$, and in the former case 
$\tilde{\chi}(1-p) = 0$.  So either $p$ or  $1-p$  is in Ker$(\tilde{\chi}) = A^{\perp \perp}$.

\section{Noncommutative peak interpolation again}

The following is a generalization of Theorem \ref{peakth} and 
Corollary \ref{peakth2}.  The  proof follows similar lines, but is a bit deeper.

\begin{theorem} \label{peakthang}   Suppose that $A$ is an operator algebra
(not necessarily approximately unital),
and that  $q$ is a closed projection in $(A^1)^{**}$.
If $b \in A$ with $b q= qb$, then $b$ achieves its distance to the
right ideal $J = \{ a \in A : q a = 0 \}$.
 If further $\Vert b q \Vert \leq 1$,
then  there exists an element $g \in {\rm Ball}(A)$ with $g q =  q g = b q$.
\end{theorem}
\begin{proof}  
If $\varphi \in A^\perp$, then
$(q a)(\varphi) = 0$ for all $a \in A$, since $q$ is weak* approximable by
elements in $A^1$, and $A^1 A \subset A$.  Thus $q$ 
satisfies the hypothesis of \cite[Proposition 3.1]{Hay},
with $X = A$.  It follows from that result that if
$J = (1-q) (A^1)^{**} \cap A$ then $d(x,J) = \Vert q x \Vert$ for all $x \in A$.

Next, let $\tilde{D} = (1-q) (A^1)^{**} (1-q)  \cap A^1$.  By \cite[Section 2]{BHN},
$\tilde{D}$ is a HSA in $A^1$, and is approximately unital.  Let $C$ be the closed
subalgebra of $A^1$ generated by $\tilde{D}, b,$ and $1$.  Then $\tilde{D}$ is an ideal
in $C$: note for example
that if $d \in \tilde{D}$ then $db = (1-q) d b = d (1-q) b =  d b (1-q)$,
so $d b \in \tilde{D}$.   Since
$\tilde{D}^{\perp \perp} = (1-q) (A^1)^{**} (1-q)$, by
\cite[Section 2]{BHN}, we have that
$1-q \in \tilde{D}^{\perp \perp} \subset C^{\perp \perp}$,
and so $q \in C^{\perp \perp}$.  Indeed
 $q$ is in the commutant of $C$,
hence in the center of $C^{\perp \perp}$.
Thus $\tilde{D}^{\perp \perp} = (1-q) C^{\perp \perp}$,
and so $\tilde{D}  =
(1-q) C^{**}  \cap C$ is an $M$-ideal in $C$.  The associated
$L$-projection $P$
 onto the subspace $\tilde{D}^\perp$ of $C^*$, is multiplication by $q$,
since multiplication by $1-q$ is the $M$-projection from $C^{**}$
onto $\tilde{D}^{\perp \perp}$.   Let $I
= C \cap A$.
It is an easy exercise that $I$ is an ideal in $C$, 
for example using the fact that $\tilde{D} A \subset A^1 A \subset A$ and
similarly $A\tilde{D}  \subset A$.  
Let $D = I \cap      
\tilde{D} = \{ x \in I : q x = 0 \}$.

If $x \in I$ and
$\varphi \in I^\perp$ then $q \varphi(x) = \lim_t \varphi(c_t x) = 0$ if $(c_t)$ is a net in
$C$ with weak* limit $q$, since $c_t x \in I$.   We will make two deductions from this.
First, $P(I^\perp) \subset I^\perp$.
 So by \cite[Proposition I.1.16]{HWW}, we have that $D = I \cap \tilde{D}$ is an $M$-ideal in $I$, hence it is proximinal in $I$.  
Second, we deduce from \cite[Proposition 3.1]{Hay} with $X = I$, that         
$d(x,D) = \Vert q x \Vert$ for all $x \in I$.
So $d(b,D) = \Vert q b \Vert = d(b,J)$, the latter from the first paragraph of the proof.
By proximinality there exists a $y \in D \subset J$ such that
$\Vert b - y \Vert = d(b,D) = \Vert q b \Vert = d(b,J)$.   We finish as before: 
setting $g = b-y$ then $q g = g q = q b$, and so on.
  \end{proof}

\begin{theorem} \label{babint}   Suppose that $A$ is an operator algebra
(not necessarily approximately unital), a subalgebra of a unital $C^*$-algebra $B$.
Identify $A^1 = A + \Cdb 1_B$.  Suppose
 that  $q$ is a closed projection in $(A^1)^{**}$.
   If $b \in A$ with $b q= qb$,
 and $q b^* b q  \leq q d$ for an invertible positive $d \in B$ which commutes with $q$,
then  there exists an element $g \in A$ with $g q =  q g = b q$, and $g^* g \leq d$.
\end{theorem}
 
The proof of this is a small modification of the last proof \cite{Bnew}.
 We will however prove here
 a one-sided variant of Theorem \ref{babint}. 

\begin{lemma} \label{peaktr}   Let  $A$ be a unital operator
algebra, and suppose that $q$ is a peak projection, the peak for
an element $a \in {\rm Ball}(A)$.   Let $B$ be a $C^*$-algebra generated by $A$.
Suppose that $b \in B$ and $|b|$ commutes with $a$.
Suppose also that $q b^* b q \leq  q d q$,
for some invertible positive $d \in B$ which commutes with $a$.
Given an open projection $u \geq q$ which commutes with $a$, and
any $\epsilon > 0$, then there exists an $n \in \Ndb$
with $\Vert  b d^{-\frac{1}{2}} a^n \Vert \leq 1 + \epsilon$, and
$\Vert  b  d^{-\frac{1}{2}} a^n (1-u) \Vert < \epsilon$.
\end{lemma}

\begin{proof}   We work in $B$; let $e = 1_{B}$ and $f = d^{\frac{1}{2}}$.
  By Lemma \ref{haye} we have
$\Vert a (e-u) \Vert = r < 1$, and so
$$\Vert   a^n (e-u) \Vert =  \Vert  (a (e-u))^n \Vert \leq r^n \to 0 .$$
  So the last inequality in the Lemma
will be easy.
 For the first, since $a$ and hence $q$ commutes with
$f, f^{-1},$ and $|b|$, we have
$$f^{-1} q b^* b qf^{-1}
= q f^{-1} b^* b f^{-1} q \leq f^{-1} qdq f^{-1}  = q ,$$ so  that
$\Vert q |b f^{-1}|^2  q \Vert
\leq 1$.    Let $b' = b f^{-1}$, so that $\Vert |b'|q \Vert
\leq 1$,
and let $p$ be a spectral projection for $[0, 1+ \epsilon)$ for $|b'|$.
Then $p$ is open, and it commutes with $f^{-1} b^* b f^{-1} = |b'|^2$ and $a^n$
(since $a$ commutes with $f^{-1} b^* bf^{-1} = |b'|^2$).  Also $p q = q$
since $\Vert |b'| q \Vert \leq 1$ (this is a nice exercise in the
Borel functional calculus, or follows easily
from \cite[Corollary 5.6.31]{KR}).   Then
$$\Vert b f^{-1}    a^n \Vert = \Vert |b'|  \, a^n \Vert = \max \{  \Vert |b'|
a^n (e- p) \Vert  ,  \Vert |b'|  \, a^n p \Vert  \}
\leq  \max \{  \Vert b'  \Vert \Vert a^n (e- p) \Vert  ,
 \Vert |b'| \, p \Vert  \},$$
which is dominated by $1 + \epsilon$
for $n$ large enough,
since by the first lines of the present proof we can choose $n$ with
$\Vert a^n (e- p) \Vert  = r^n < (1 + \epsilon)/ \Vert b'  \Vert$.
 \end{proof}

\begin{corollary} \label{peakth4}
Suppose that $A$ is an  operator
algebra (possibly not approximately unital), and
that $q$ is a projection in $(A^1)^{**}$ such that $q$ is the
peak of an element $a \in {\rm Ball}(A^1)$.
Suppose that $b \in A$, such that $|b|$ commutes with $a$.
Let $B$ be a unital $C^*$-algebra containing $A$,
and suppose also that $q b^* b q \leq  q d q$,
for some invertible positive  $d \in B$ which commutes with $a$.
Then  there exists an element $x \in A$ with $x^* x \leq d$
and $x q
 = b q$. \end{corollary}

\begin{proof}  We follow the classical idea due to Bishop
(see II.12.5 in \cite{Gam}), with some variations of our own.
Set $f = d^{\frac{1}{2}}$.
  Letting $\epsilon = \frac{1}{4}$,
by Lemma \ref{peaktr} with respect to $A + \Cdb 1_B$,
there exists $m_1 \in \Ndb$
such that if $g_1 = b f^{-1}  a^{m_1}$
then $\Vert g_1 \Vert \leq \frac{5}{4}$, and $g_1 q = b f^{-1} q
= b q f^{-1}$.  Let $u_n$ be the spectral projection
for $[0, 1 + \frac{1}{2^{n+1}})$ for $|b'|$, for $n \geq 2$,
where $b'$ is as in the last result. Note that $(u_n)$ is a decreasing sequence of open projections, and $u_n \geq q$
and $u_n$ commutes with $a$, as in the proof of  Lemma \ref{peaktr}.
By Lemma \ref{peaktr} again, we can choose an  increasing sequence
of positive  integers $(m_n)$  with
$\Vert  b f^{-1}  a^{m_n} (1 - u_n) \Vert \leq \frac{1}{2}$ and
$\Vert b f^{-1}  a^{m_n} \Vert \leq  1 +  \frac{1}{2^{n+1}}$.
 Let $g_k = b f^{-1} a^{m_k} = b a^{m_k} f^{-1}$, and let
$g  = \sum_{k=1}^\infty \frac{g_k}{2^k},$ and $x = g f
= \sum_{k=1}^\infty \frac{g_k f}{2^k} \in A$.
 We have $g q =  \sum_{k=1}^\infty \frac{g_k q}{2^k} = b f^{-1}  q
= b q f^{-1},$  and so $$x q = gf q = g q f = bq ,$$
as desired.     We
also have $$\Vert  g (1 - u_2) \Vert
\leq  \frac{5}{8} + \sum_{k=2}^\infty \frac{\Vert   g_k (1 - u_k) \Vert}{2^k}
\leq   \frac{5}{8} + \frac{1}{2} \sum_{k=2}^\infty \frac{1}{2^k} =  \frac{7}{8} \leq 1.$$
  If $r$ is a projection dominated by $u_n$ for every $n$, then $$\Vert g_n r \Vert \leq \Vert g_n u_m \Vert \leq \Vert
|b'|  u_m \Vert \leq 1 + \frac{1}{2^{m+1}}, \qquad m \in \Ndb ,$$ so that $\Vert g_n r \Vert \leq 1$.
Hence $\Vert g r \Vert \leq 1$.   If $s = u_{n-1} - u_n$ for some $n \geq 3$, then for $k \geq n$ we have $$\Vert g_k s \Vert
\leq \Vert g_k (1- u_n) \Vert
\leq \Vert g_k (1 - u_k) \Vert
\leq \frac{1}{2}.$$  If $k < n$ then $$\Vert g_k s \Vert
\leq \Vert |b'| \, u_{n-1} \Vert
\leq 1 +  \frac{1}{2^{n}}.$$
Thus $$\Vert g s \Vert \leq \sum_{k=1}^{n-1} \, \frac{ \Vert g_k s \Vert}{2^k} +  \sum_{k=n}^\infty \,  \frac{1}{2^{k+1}}
\leq (1 +  \frac{1}{2^{n}})(1- \frac{1}{2^{n-1}}) + \frac{1}{2^{n}}  = 1 - \frac{1}{2^{2n-1}} \leq 1.$$
Since $\Vert g \Vert = \Vert |b'| \sum_{k=1}^\infty \frac{a^{m_k}}{2^k} \Vert$, and
the $(u_k)$ commute with $|b'|$ and $a^{m_k}$,  we see  that
$\Vert g \Vert$ is dominated by the maximum of
$\Vert g (1-u_2) \Vert , \Vert g \, (\wedge_k u_k ) \Vert$,
and $\sup_{n \geq 3} \, \Vert g \, (u_{n-1} - u_n) \Vert$, each of which is $\leq 1$.
So $g^* g \leq 1$ and $x^* x = f g^* g f\leq f^2 = d$.    \end{proof}

{\bf Remarks.} 1) \   Corollary \ref{peakth4},
Theorem \ref{peakthang},  and the matching results in Section 2,
lead to  `Rudin-Carleson theorems' of the type in
\cite[Proposition 3.4]{Hay}.

\smallskip

2) \ Corollary \ref{peakth4}  is not true if one drops the condition
that $d$ commutes with $a$.  For example, take $A$ to be the lower triangular
matrices in $M_2$, $q = a = E_{11}$, and $d = \epsilon E_{22} + |x
\rangle \langle x |$, where $x = [1 \; \; -1]$.  We do not know
if Corollary \ref{peakth4}  is true if one replaces the condition
that $d$ commutes with $a$ with $d$ commuting with $q$.


\section{Compact projections}

In a recent paper
\cite{BNII} the first author and Neal developed (generalizing work of Akemann and coauthors, see e.g.\ \cite{AP}) the
theory of compact projections in an approximately unital
operator algebra $A$.  
We defined a projection $q \in A^{**}$ to be {\em compact relative to}
 $A$ if it is closed  and $q x = q$
for some $x \in {\rm Ball}(A)$. 

\begin{lemma}  \label{coco} Let $A$ be an approximately unital
operator algebra.    A closed projection $q \in A^{**}$ is {\em compact} if $q x = q$
for some $x \in A$.
\end{lemma}

\begin{proof}  One direction is obvious.
For the other, by \cite[Theorem 2.2]{BNII} we may assume that $A$ is a $C^*$-algebra.
We have $x^* q x = q$.  There is a net $(a_t)$ in $A^1$ which decreases to $q$
(since there is a net increasing to $1-q$),
and if $b = x^* a_t x$ then  $q \leq b \in A_+$, which is one of the
standard definitions of a compact projection (see e.g.\ 
the lines above \cite[Lemma 2.7]{AP}).
\end{proof}

We thank C. Akemann for communicating to us the idea for Lemma \ref{coco} in the $C^*$-algebra case
(it does not seem to appear in the literature).   

Turning to the case of a (not necessarily approximately unital) operator algebra $A$,
 there are   
many possible notions of compactness which come to mind.   Fortunately 
these collapse to two
notions, one stronger than the other.  
We will simply say that a  projection $q \in A^{**}$ is {\em compact} (relative to $A$, or
with respect to $A$)
if $q$ is closed in $(A^1)^{**}$ with respect to $A^1$.  This 
is the same as  
$q \in A^{\perp \perp}$ being compact with respect to $B^{**}$ for a containing $C^*$-algebra $B$,
by \cite[Theorem 2.2]{BNII}  and the $C^*$-algebra case of that result.

\begin{theorem}  \label{compgiv}  Suppose that $A$ is an  operator algebra
(not necessarily approximately unital), and that $q \in A^{**}$ is a projection.  
 The following
  are equivalent:
  \begin{enumerate} \item  $q$ is compact with respect to $A$ in the sense just defined.
 \item  $q$ is closed with respect to $A^1$ and there exists 
$a \in {\rm Ball}(A)$ with $aq = qa = q$.
\item  $q$ is a decreasing (weak*) limit of projections of
form $u(a)$ for $a \in {\rm Ball}(A)$.
\end{enumerate}  \end{theorem} 

\begin{proof}  (1) $\Rightarrow$ (2) \  We have by 
e.g.\ \cite[Theorem 3.4 (1)]{BNII} that $q$ is a decreasing limit of projections of
form $u(x)$ for $x \in {\rm Ball}(A^1)$.  Suppose that each $x$ is of the form  $1 + a$ for some
$a \in A$.  Then $x^n$ is of this form for each $n$, and if $\chi : A^1 \to A^1/A$ is the 
canonical quotient then $1 = \chi^{**}(x^n) \to \chi^{**}(u(x))$, so that 
$\chi^{**}(u(x)) = 1$.  Hence $\chi^{**}(q) = 1$, a contradiction.
Thus at least one such  $x = a + c1$ for some $a \in A, c \in \Cdb \setminus \{1 \}$.  Then 
$$q = q u(x) =  q u(x) x = q x = q a + cq, $$ so that $q z = z = zq$ where 
$z = \frac{1}{1-c} a$.   By Theorem \ref{peakthang} with $d = 1$, there exists 
$a \in {\rm Ball}(A)$ with $aq = qa = q$.

(2) $\Rightarrow$ (3) \  Follows from the proof of  \cite[Theorem 3.4 (1)]{BNII}. 

(3)  $\Rightarrow$ (1) \ Clear for example from 
 \cite[Theorem 3.4 (1)]{BNII} since $u(a)$ is a peak projection
 with respect to $A^1$.   
\end{proof}
 
{\bf Remark.}  The condition in (3) is equivalent to $q$ being an infimum of peak projections, as in \cite[Theorem 3.4 (1)]{BNII},
and with the same proof (note that $u(a) \wedge u(b) = u(\frac{a+b}{2})$ as in that paper, for $a, b \in {\rm Ball}(A)$ such that
$u(a)$ and $u(b)$ are projections).

\begin{corollary}  \label{noct} Let $A$ be a (not necessarily approximately unital)
 operator algebra.  If $q \in A^{**}$ is compact then $q$ is a weak* limit of a 
net $(a_t)$ in ${\rm Ball}(A)$  with $a_t q = q a_t = q$ for all $t$.
\end{corollary}  \begin{proof}  
Choose $a_t \in {\rm Ball}(A)$
with $u(a_t) \searrow q$ (see Theorem \ref{compgiv} (3)).  Then $q a^n_t = q u(a_t) a^n_t = q u(a_t) = q$.
Since the double weak* limit $\lim_t \lim_n a^n_t = \lim_t u(a_t) =  q$,
 a reindexing of $(a^n_t)$ is a net of contractions $y_t \to q$ weak*
 with $qy_t = y_t q = q$.
\end{proof}

{\bf Remark.}   The fact in Corollary \ref{noct} was stated in 
\cite{BNII} for approximately unital algebras 
before Theorem 2.1 there.  Unfortunately there seems to be a typo there:
the  construction does not produce elements in ${\rm Ball}(A)$ in general.  This is easily 
fixed though by choosing the $e_t$ there in $\frac{1}{2} {\mathfrak F}_A$ by Read's theorem
 \cite{Read,Bnew}.   
 
\begin{proposition} \label{cpeakthang}  Suppose that
 $A$ is a (not necessarily approximately unital)
operator algebra.
  \begin{enumerate} \item  If $q$ is a projection in
$A^{**}$ and $q = u(x)$, 
for some $x \in  {\rm Ball}(A^1)$, then  $q = u(a)$, 
for some $a \in  {\rm Ball}(A)$.
\item  If $A$ is separable then the 
compact projections in $A^{**}$ are precisely the projections in $A^{**}$ of the form $u(a)$,
for some $a \in  {\rm Ball}(A)$.
\end{enumerate}  \end{proposition}  

\begin{proof}  (1) \  Since $u(x)$ is closed in $A^1$, $q$ is compact for $A$, and so by the previous result
$q = qb$ for some $b \in {\rm Ball}(A)$.  By the last lines of proof of \cite[Theorem 3.4 (2)]{BNII}
we also have $q = u(bx)$.  

(2) \ This follows from (1) and the fact from  \cite[Theorem 3.4 (2)]{BNII}
that since $A^1$ is separable any projection compact with respect to $A^1$
equals $u(x)$, 
for some $x \in  {\rm Ball}(A^1)$.   
\end{proof}

We  define 
a {\em ${\mathfrak F}$-peak projection} for $A$ to be $u(x)$, the weak* limit
of the powers $x^n$, for 
some $x \in \frac{1}{2} {\mathfrak F}_A$.   See 
 \cite[Corollary 3.3]{BNII} for the fact that this
weak* limit exists and is a projection, which is 
nonzero if $\Vert x \Vert = 1$.  We define a
projection in $A^{**}$ to be  ${\mathfrak F}$-compact if it is a decreasing 
limit of  ${\mathfrak F}$-peak projections.  If $A$ is approximately unital then
the compact projections relative to $A$ as defined above,
the ${\mathfrak F}$-compact projections, and 
the compact projections  in \cite{BNII},
are the same, by Theorems 2.2 and  3.4 in that reference.  However 
if $A$ is not approximately unital then there may exist 
compact projections in the sense above which are not ${\mathfrak F}$-compact projections.
(This is the case in the 
example at the end of Section 4, where
${\mathfrak F}_A = (0)$, yet the copy of the characteristic
function of $\{ 1 \}$ in $C(\Tdb)^{**}$ equals $u(f)$ for $f = z (z+1)/2 \in A$.)

\begin{proposition} \label{coh}  If 
 $A$ is any  
operator algebra, then 
\begin{itemize}  
\item [(1)]  A projection in $A^{**}$  is
 ${\mathfrak F}$-compact in the sense above iff it is a compact projection
in the sense of {\rm \cite{BNII}} for $A_H$.   
\item [(2)]  A projection in $A^{**}$  is a 
 ${\mathfrak F}$-peak projection in the sense above iff it is a peak projection
in the sense of {\rm \cite{BNII}} for $A_H$, and iff it equals $u(a)$ for some
$a \in {\rm Ball}(A_H)$.  
\item [(3)] If $A$ is separable then every 
 ${\mathfrak F}$-compact projection in $A^{**}$ is a  ${\mathfrak F}$-peak projection.  
 \end{itemize} \end{proposition} \begin{proof}  Most of these
use the fact from Theorem
\ref{sps} that
${\mathfrak F}_A = {\mathfrak F}_{A_H}$, together with facts stated in the introduction,
or above the Proposition, about peak projections.  

(2) \ By Theorem 3.4 in \cite{BNII},
 $q$ is a compact projection in the sense of \cite{BNII} for $A_H$,  iff it is a decreasing
limit of terms of the form $u(x)$ for $x \in \frac{1}{2} {\mathfrak F}_{A_H} = 
\frac{1}{2} {\mathfrak F}_A$.  That is, iff  it is  ${\mathfrak F}$-compact.

 (1) \ By definition $q$ is ${\mathfrak F}$-peak iff $q = u(x)$ for $x \in \frac{1}{2} {\mathfrak F}_A
= \frac{1}{2} {\mathfrak F}_{A_H}$,
which by what we said in the 
introduction is equivalent to the other conditions.

(3) \ This is obvious  from (1), (2), and \cite[Theorem 3.4 (2)]{BNII},
 since in this case $A_H$ is separable too.   
\end{proof}

It now is  a simple matter to generalize to algebras with no approximate identity 
other results from \cite{BNII} concerning compact 
projections, using the two generalizations of compactness above.
For the second, ${\mathfrak F}$-compactness, this is usually easier than for the first. 
We mention for example that results 2.3, 2.4, and 5.1 from that paper are clearly true for 
what we have called compact projections above,
for algebras with no approximate identity (with all occurrences of 
the words `approximately unital' removed).  The proofs are unchanged.  Also both 
noncommutative Urysohn lemmas from \cite{BNII} do  generalize:
  
\begin{theorem} \label{ncuch} {\rm (Noncommutative Urysohn lemma for
general operator algebras) } \ Let $A$ be a
(not necessarily approximately unital) operator algebra, a subalgebra of a $C^*$-algebra $B$,
and let $q$ be a compact projection  in $A^{**}$.  Then
\begin{itemize}
\item [(1)]  For any open
projection $p \in B^{**}$ with $p \geq q$, and any $\epsilon > 0$,
there exists an $a \in {\rm Ball}(A)$ with
$a q = q$ and $\Vert a (1-p) \Vert < \epsilon$
and $\Vert (1-p)  a \Vert < \epsilon$.
\item [(2)]  For any open
projection $p \in A^{**}$ with $p \geq q$, 
there exists $a \in {\rm Ball}(A)$ with $q = q a, a= ap$.   
\end{itemize}
\end{theorem}

 \begin{proof}
(2) \ Apply  \cite[Theorem 2.6]{BNII} in $A^1$:  if $a \in A^1, p \in 
A^{\perp \perp}$ and $a p = a$, then $a \in 
 A^{\perp \perp} \cap A^1 = A$ (since $A^{\perp \perp}$ is an ideal in $(A^1)^{**}$).

(1) \ The  proof of \cite[Theorem 2.1]{BNII} works, 
taking $(y_t)$ to be the net in  \ref{noct}.
 \end{proof}   

We say that a right ideal $J$ in $A$ is regular (resp.\ $1$-regular) if
there exists an $x \in A$ (resp.\ $x \in {\rm Ball}(A)$) with $(1-x) A \subset J$.

\begin{proposition} \label{1reg2}  If $q$ is a compact projection in $A^{**}$
then  the right ideal $\{ a \in A : qa = 0 \}$ is 1-regular.  
  \end{proposition} \begin{proof}
If $q a = q$ with
$a \in {\rm Ball}(A)$ then $(1-a)A \subset (1-q) (A^1)^{**} \cap A$.
 \end{proof}

\begin{proposition} \label{1reg}  An r-ideal $J$ in an approximately unital
operator algebra $A$ is regular iff it is $1$-regular,
and iff the complement of the support projection of $J$ is a compact projection. \end{proposition} \begin{proof}
Let $p$ be  the support projection of $J$, and suppose that $x \in A$.
Then $(1-x) A \subset J$ iff
$p (1-x) A  = (1-x) A$ iff $p (1-x) = (1-x)$ iff $p^\perp  x = p^\perp$.
The result now follows from Lemma \ref{coco}.
 \end{proof}

{\bf Remark.}  1) \ In the last proposition, one may further choose
$x \in \frac{1}{2} {\mathfrak F}_A$ if $A$ is approximately unital \cite{BNII}.

2) \  Every r-ideal in a unital operator algebra $A$ is 1-regular by \cite[Theorem 1.2]{BRead}, but this
is not true if $A$ has a cai (this may be seen using Proposition \ref{1reg} and the
fact that some algebras have no compact projections \cite{BNII}).  We also remark that the first `iff' in
Proposition \ref{1reg} is
false with `approximately unital' removed (as one may see in three dimensional algebras of upper triangular matrices).

\begin{corollary}  \label{noc} Let $A$ be a (not necessarily approximately unital)
 operator algebra.  The following are equivalent:
\begin{enumerate} \item There exist no nonzero compact projections in $A^{**}$, \item
The spectral radius $r(x) < \Vert x \Vert$ for all $x \in A$, \item  The numerical radius
$\nu(x) < \Vert x \Vert$ for all $x \in A$, \item $\Vert 1  +x \Vert < 2$ for all $x \in {\rm Ball}(A)$,
\item $\overline{(1-x)A} = A$ for all $x \in {\rm Ball}(A)$.  \end{enumerate} \end{corollary}

\begin{proof}  If any of these five conditions hold, then it is easy to see that
$A$ contains no projections.
By \cite[Theorem 3.5 and Proposition 3.7]{ABS}
we have that (2), (3), (4), (5) are all equivalent, and are also  equivalent to
every element of Ball$(A)$ being quasi-invertible.  However
if $q$ was a nonzero compact projection for $A$, then $\{ a \in A : q a = 0 \}$ is not $A$, 
and so $\overline{(1-x)A} \neq A$ by Proposition \ref{1reg2} and its proof, where 
$q = qx$ with $x  \in {\rm Ball}(A)$.

Conversely, 
if $\Vert 1  +x \Vert < 2$ for $x \in {\rm Ball}(A)$, then $q = u(\frac{1+x}{2})$ is
a nonzero projection by e.g.\  \cite[Corollary 3.3]{BNII}, is in $A^{**}$ by e.g.\ 
\cite[Proposition 6.9(i)]{BHN}, and is closed in $(A^1)^{**}$ like all
peak projections.
\end{proof}

\bigskip

{\em Acknowledgements.}  We thank Sonia Sharma for several discussions, and Damon
Hay for reading a draft of the paper.  We thank M. Neal for discussions on
compactness of projections and the Urysohn lemma.
We are deeply grateful 
to the referee for carefully  reading the manuscript and making 
several  helpful stylistic suggestions,
 and especially for providing,  
and allowing us to include,
 the beautiful proof of what is now  Theorem \ref{sect}.   Also
the current proof of Theorem \ref{cls} is due 
to the referee, and is shorter than our original one.

\end{document}